\documentclass{amsproc}
\usepackage{amssymb}
\usepackage{amsmath}
\usepackage{epsfig}
\usepackage{graphicx}
\usepackage[all]{xy}
\usepackage{color}
\usepackage{enumitem,kotex}  
\usepackage{setspace}
\usepackage{amsthm,amsfonts,amssymb,epsfig,graphics,amsmath,amsbsy,subfigure}

\newcommand{\intL}{\int\limits}

\newcommand{\RR}{\mathbb{R}}
\newcommand{\xx}{\mathbf{x}}
\newcommand{\uu}{\mathbf{u}}

\newtheorem{thm}{Theorem}

\newtheorem{lem}[thm]{Lemma}
\newtheorem{prop}[thm]{Proposition}

\begin{document}
\title[Inversion for variable-speed wave equations on Bounded Domains]{Explicit inversion for variable-speed wave equations on bounded domains}
\author{Sunghwan Moon}
\address{Department of Mathematics, College of Natural Sciences, Kyungpook National University, Daegu 41566, Republic of Korea}
\curraddr{}
\email{sunghwan.moon@knu.ac.kr}
\thanks{The work  of S. Moon was supported by NRF-2022R1C1C1003464 and RS-2023-00217116.}
\author{Ihyeok Seo}
\address{Department of Mathematics, Sungkyunkwan University, Suwon 16419, Republic of Korea}
\curraddr{}
\email{ihseo@skku.edu}
\thanks{The work  of I. Seo was supported by RS-2022-NR069672.}

\subjclass[2020]{Primary 35L05;  35R30; 	92C55}

\keywords{keywords: photoacoustic, thermoacoustic tomography, wave equation, inversion}

\date{}

\dedicatory{}

%
\begin{abstract}
We study the reconstruction of the initial pressure $f(\xx)=p(\xx,0)$ for the wave model
\[
\partial_t^2 p(\xx,t)=c(\xx)\triangle_{\xx}p(\xx,t)\qquad (\xx,t)\in\Omega\times[0,\infty),
\]
posed on a bounded domain $\Omega$ with variable sound speed $c(\cdot)$. From time-resolved boundary measurements, we consider two settings: (i) measurement of $p|_{\partial\Omega\times[0,\infty)}$ under a Robin boundary condition $p+\alpha\,\partial_\nu p=0$ on $\partial\Omega\times[0,\infty)$ with $\alpha\gneq 0$, and (ii) measurement of $\partial_\nu p|_{\partial\Omega\times[0,\infty)}$ under a Dirichlet boundary condition $p=0$ on $\partial\Omega\times[0,\infty)$. Within a unified framework, we present \emph{explicit formulas} that recover the spectral coefficients $\langle f,\phi_k^B\rangle$ of $f$ with respect to the eigenfunction bases of the operator $-c(\cdot)\triangle_{\xx}$ for boundary types $B\in\{D,R\}$. The framework integrates variable sound speed with Dirichlet/Robin boundary conditions in a single setting, enabling direct coefficient-level recovery from boundary data.
\end{abstract}

\maketitle
 \section{Introduction}

Recently, many tomography researchers have considered combining different physical signal types to leverage their complementary strengths while mitigating individual limitations. The most successful example is photoacoustic tomography (PAT), also known as optoacoustic or thermoacoustic tomography. PAT is an emerging technology for noninvasive medical imaging based on the photoacoustic effect discovered by Bell in 1880 \cite{bell80}, wherein absorption of electromagnetic (EM) energy (e.g., light or radio-frequency illumination) induces thermoelastic expansion and generates acoustic waves. Ultrasound imaging typically provides high spatial resolution but low contrast, whereas EM illumination offers strong molecular contrast yet low resolution. PAT integrates these advantages, using optical contrast to seed an acoustic wave that can be recorded with ultrasound detectors, thus achieving high-contrast, high-resolution imaging.

In PAT, a short, non-ionizing EM pulse irradiates the object approximately uniformly. The absorbed energy produces thermoelastic expansion of tissue and, consequently, an acoustic pressure wave that propagates through the object and is recorded by detectors placed around it. Although ultrasound images alone may suffer from low contrast, it is well known that cancerous tissue tends to absorb significantly more energy than healthy tissue. Hence, the strength of the acoustic source encodes the spatial distribution of EM absorption and can provide sensitive contrast for early detection.

We model the acoustic pressure $p$ by the wave equation in a bounded domain $\Omega$ with smooth boundary $\partial \Omega$,
\begin{equation}\label{eq:pdeofpat}
\partial_t^2 p(\xx ,t)=c(\xx)\triangle_{\xx }p(\xx ,t)\qquad(\xx ,t)\in\Omega\times[0,\infty),
\end{equation}
where the sound speed $c\in C^\infty(\bar \Omega)$ is strictly positive, i.e., $c_m\le c(\xx)\le c_M$ for some constants $0<c_m\le c_M<\infty$ and $\bar \Omega$ is the closure of $\Omega$. Because biological sound speeds are relatively small and practical systems image bounded objects, it is natural to work on a bounded $\Omega$. We consider measurements acquired on the boundary $\partial\Omega$ by detectors over time, and the central mathematical task is to determine the initial source
\[
f(\xx)=p(\xx,0)
\]
from boundary data. Two prototypical data–boundary configurations arise:
\begin{itemize}
\item[1)] $p$ is measured on $\partial \Omega$ while the wave satisfies a Robin boundary condition
\[
p+\alpha\,\partial_\nu p=0 \quad\text{on } \partial \Omega\times[0,\infty),\qquad \alpha\gneq 0;
\]
\item[2)] $\partial_\nu p$ is measured on $\partial \Omega$ while the wave satisfies a Dirichlet boundary condition
\[
p=0 \quad\text{on } \partial \Omega\times[0,\infty).
\]
\end{itemize}
Here $\nu$ denotes the outward unit normal to $\partial\Omega$, so that $\partial_\nu=\nu\cdot\triangledown$ is the normal derivative.

The recovery of $f$ from boundary measurements has been extensively studied in the constant-speed case; see, for example, \cite{ammaribjk10,ammarigjn13,anastasiozmr07,moonip18,bukhgeimk78,dok18,dreierh20,finchhr07,finchpr04,haltmeiersbp04,moonjop16,kostlifbw01,kuchment14book,kunyansky12,sandbichlerkbbh15,xuw05,moonzangerlip19}. Inhomogeneous media have also been considered from different angles; e.g., \cite{ammarikk12} studied related inverse problems in inhomogeneous settings without the full wave equation, and several works analyzed variable-speed wave propagation on the whole space, including uniqueness and reconstruction aspects \cite{agranovskyk07,liuu15,knoxm20,stefanovu09}. For instance, \cite{agranovskyk07} derived an analytic formula for spherical acquisition, \cite{knoxm20,stefanovu09} analyzed recovery of $f$ with known $c $, and \cite{liuu15} gave sufficient conditions for unique determination of both $c$ and internal sources.

In contrast to whole-space studies or constant-speed models, we investigate the bounded-domain problem with variable speed $c$ under physically relevant boundary conditions on a bounded domain $\Omega$ with smooth boundary. 
Dirichlet and Neumann boundaries are routinely used to model sound-soft and sound-hard interfaces, respectively, while Robin conditions represent impedance-type behavior interpolating between them \cite{claytone77,holmank15,sochacki88}. 

Stefanov and Yang \cite{stefanovy15} develop sharp time-reversal in closed domains with reflective boundaries, relying on Neumann-series constructions rather than coefficient-wise identities. Nguyen and Kunyansky \cite{nguyenk16} introduce a dissipative time-reversal under mixed boundary operators to stabilize backward propagation (often with exponential convergence as the recording time increases), but again do not furnish explicit spectral-coefficient formulas. Kunyansky, Holman and Cox \cite{kunyanskyhc13} analyze a rectangular resonant cavity with essentially constant sound speed and propose a fast iterative scheme tailored to that geometry. Moon, Hur, and Moon \cite{moonsiims23} assume a spherical acquisition surface with radial symmetry and study the SVD of the forward operator, yielding formulas specialized to the sphere. By contrast, this paper treats an arbitrary bounded domain with spatially varying sound speed and accommodates both Dirichlet and Robin boundary types within a single spectral framework. We derive explicit identities that recover the spectral coefficients directly from time-resolved boundary traces, providing a geometry-agnostic, coefficient-level inversion that extends beyond spherical settings and—to our knowledge—offers the first unified explicit recovery for both boundary conditions in variable-speed media.

For later reference, we define the wave-forward operators. Let $\mathcal W_R f(\xx,t)$ denote the solution of \eqref{eq:pdeofpat} with initial data
\begin{equation}\label{eq:initial}
p(\xx ,0)=f(\xx ) \quad\mbox{and}\quad \partial _t p(\xx ,0)=0 \qquad (\xx\in\Omega),
\end{equation}
subject to the Robin boundary condition
\[
\mathcal W_R f+\alpha\,\partial_\nu \mathcal W_R f=0 \quad\text{on }\partial\Omega\times[0,\infty),\qquad \alpha\gneq 0.
\]
Similarly, let $\mathcal W_D f(\xx,t)$ denote the solution of \eqref{eq:pdeofpat}–\eqref{eq:initial} with the Dirichlet boundary condition $\,\mathcal W_D f=0$ on $\partial\Omega\times[0,\infty)$. We formulate two inverse problems on the weighted $L^2$ space $L^2(\Omega,c(\xx)^{-1}{\rm d}\xx)$:
\[
\textbf{[Robin data]} \;\; \text{Reconstruct } f\in L^2(\Omega,c(\xx)^{-1}{\rm d}\xx) \text{ from } \mathcal W_R f\big|_{\partial\Omega\times[0,\infty)};
\]
\[
\textbf{[Dirichlet data]} \;\; \text{Reconstruct } f\in L^2(\Omega,c(\xx)^{-1}{\rm d}\xx) \text{ from } \partial_\nu \mathcal W_D f\big|_{\partial\Omega\times[0,\infty)}.
\]
The main results provide explicit coefficient-recovery formulas for these settings under the above assumptions, and they are stated and proved in the next sections.

\section{Preliminaries}\label{sec:prelim}

Let $\Omega\subset\mathbb{R}^n$ be a bounded domain with a smooth boundary and outward unit normal $\nu$, and let the sound speed $c\in L^\infty(\Omega)$ satisfy $0<c_m\le c(\xx)\le c_M<\infty$. We work in the weighted space $L^2(\Omega,c(\xx)^{-1}{\rm d}\xx)$ with inner product
\[
\langle f,g\rangle:=\int_\Omega f(\xx)\,g(\xx)\,c(\xx)^{-1}\,{\rm d}\xx .
\]

We consider the two elliptic boundary value problems that will generate orthonormal bases adapted to the Robin and Dirichlet boundaries:
\begin{align}
\label{eq:robin-Poisson}
\begin{cases}
-\,c(\xx)\Delta u(\xx)=f(\xx) & \text{in }\Omega,\\
u+\alpha\,\partial_\nu u=0 & \text{on }\partial\Omega,\qquad \alpha>0,
\end{cases}
\qquad
\begin{cases}
-\,c(\xx)\Delta u(\xx)=f(\xx) & \text{in }\Omega,\\
u=0 & \text{on }\partial\Omega.
\end{cases}
\end{align}
By a standard Lax--Milgram argument (using the boundary coercivity inequality for the Robin case and the Poincar\'e inequality for the Dirichlet case), for every $f\in L^2(\Omega,c^{-1}(\xx){\rm d}\xx)$ there exists a unique weak solution $u$ in $H^1(\Omega)$ (Robin) or $H_0^1(\Omega)$ (Dirichlet).

Define the solution operator $T:L^2(\Omega,c^{-1}(\xx){\rm d}\xx)\to H(\Omega)$ by $Tf=u$, where $u$ solves \eqref{eq:robin-Poisson} in the relevant space $H(\Omega)\in\{H^1(\Omega),H_0^1(\Omega)\}$. The mapping $T$ is bounded, and by compactness of the embedding $H(\Omega)\hookrightarrow L^2(\Omega)$ it induces a compact operator on $L^2(\Omega,c^{-1}(\xx){\rm d}\xx)$. Moreover, $T$ is self-adjoint with respect to $\langle\cdot,\cdot\rangle$.

By the spectral theorem for compact self-adjoint operators, $L^2(\Omega,c^{-1}(\xx){\rm d}\xx)$ admits orthonormal eigenfunction bases $\{\phi^R_l\}_{l\ge1}$ and $\{\phi^D_k\}_{k\ge1}$ associated with the Robin and Dirichlet realizations of $-c(\xx)\Delta$, i.e.
\[
c(\xx)\Delta \phi^R_l(\xx) + \lambda_{R,l}^2\,\phi^R_l(\xx)=0 \;\text{ in }\Omega, 
\quad \phi^R_l+\alpha\,\partial_\nu\phi^R_l=0 \;\text{ on }\partial\Omega,
\]
\[
c(\xx)\Delta \phi^D_k(\xx) + \lambda_{D,k}^2\,\phi^D_k(\xx)=0 \;\text{ in }\Omega, 
\quad \phi^D_k=0 \;\text{ on }\partial\Omega,
\]
with strictly positive eigenvalues $\lambda_{R,l}^2,\lambda_{D,k}^2>0$. 

Consequently, every $f\in L^2(\Omega,c^{-1}(\xx){\rm d}\xx)$ has the (convergent) expansions
\[
f=\sum_{l=1}^\infty \langle f,\phi^R_l\rangle\,\phi^R_l
\qquad\text{and}\qquad
f=\sum_{k=1}^\infty \langle f,\phi^D_k\rangle\,\phi^D_k .
\]
The statements above are standard; short proofs and precise references are given in Appendix~A.

For the initial data $(p,\partial_t p)|_{t=0}=(f,0)$ in \eqref{eq:initial}, there exists the unique weak solution with boundary type $B\in\{D,R\}$, which
 admits the convergent expansion
\begin{equation}\label{eq:series}
p_B(\xx,t)=\sum_{k=1}^\infty \cos\!\big(\lambda_{B,k} t\big)\,
( f,\phi_k^B)_{L^2_c}\,\phi_k^B(\xx).
\end{equation}
 If $f\in C^\infty(\Omega)$ and $c\in C^\infty(\bar\Omega)$, then $p_B\in C^\infty([0,\infty);L^2(\Omega,c(\xx)^{-1}{\rm d}\xx))$ (more detailed, see \cite[Chapter~7]{evans10}). 


\section{Inversion Procedure}
We provide a method to obtain the initial function $f$ from the two types of boundary data of the wave equation on the bounded  domain $\Omega$ with smooth boundary.
We first record the following elementary identity.

\begin{lem}\label{lem:haltmeier}
For $k,l\geq 1$, we have
\[
\big(\lambda_{R,l}^2-\lambda_{D,k}^2\big)\,\langle\phi^R_l,\phi^D_k\rangle
=\int_{\partial \Omega}\phi^R_l(\xx)\,\partial_\nu \overline{\phi^D_k(\xx)}\,{\rm d}\sigma(\xx).
\]
\end{lem}

\begin{proof}
By the eigenvalue equations $-\Delta_{\xx}\phi^R_l=\lambda_{R,l}^2 c(\xx)^{-1}\phi^R_l$ and
$-\Delta_{\xx}\phi^D_k=\lambda_{D,k}^2 c(\xx)^{-1}\phi^D_k$,
\[
\begin{aligned}
\big(\lambda_{R,l}^2-\lambda_{D,k}^2\big)\langle\phi^R_l,\phi^D_k\rangle
&=\int_{\Omega}\!\!\big(\lambda_{R,l}^2\phi^R_l\overline{\phi^D_k}-\lambda_{D,k}^2\phi^R_l\overline{\phi^D_k}\big)c^{-1}\,{\rm d}\xx\\
&=\int_{\Omega}\!\!\big(-\overline{\phi^D_k}\,\Delta_{\xx}\phi^R_l+\phi^R_l\,\Delta_{\xx}\overline{\phi^D_k}\big)\,{\rm d}\xx\\
&=\int_{\partial \Omega}\!\!\big(-\overline{\phi^D_k}\,\partial_\nu \phi^R_l+\phi^R_l\,\partial_\nu \overline{\phi^D_k}\big)\,{\rm d}\sigma(\xx),
\end{aligned}
\]
where the last step uses Green's second identity. Since $\phi^D_k|_{\partial\Omega}=0$, the first boundary term vanishes and the claim follows.
\end{proof}

\subsection{Inverse Problem with Robin boundary condition}\label{sec:robin}
In this subsection, we recover $f\in L^2(\Omega,c(\xx)^{-1}{\rm d}\xx)$ from the boundary traces of $\mathcal W_R f$. By linearity and the completeness of $\{\phi^D_k\}_{k=1}^\infty$ in $L^2(\Omega,c(\xx)^{-1}{\rm d}\xx)$, it suffices to reconstruct the coefficients $\langle f,\phi_k^D\rangle$.

\begin{thm}\label{thm:co1}
For $f\in L^2(\Omega,c(\xx)^{-1}{\rm d}\xx)$ and every $k\ge 1$, we have 
\[
\langle f,\phi^D_k\rangle
= -\,\lambda_{D,k}^{-1}\!
\lim_{\epsilon\to0}\int_0^\infty\!\!\frac{1}{2}\,e^{-\epsilon t}\sin(\lambda_{D,k}t)
\!\left(\int_{\partial \Omega} \mathcal W_R f(\uu,t)\,\partial_\nu \overline{\phi^D_k(\uu)}\,{\rm d}\sigma(\uu)\right)\!{\rm d}t.
\]
\end{thm}

\begin{proof}
By linearity, it is enough to prove the formula for $f=\phi_l^R$. Since $\mathcal W_R\phi_l^R(\xx,t)=\phi_l^R(\xx)\cos(\lambda_{R,l}t)$,
Lemma~\ref{lem:haltmeier} gives
\[
\int_{\partial \Omega}\mathcal W_R\phi_l^R(\xx,t)\,\partial_\nu \overline{\phi^D_k(\xx)}\,{\rm d}\sigma(\xx)
=\big(\lambda_{R,l}^2-\lambda_{D,k}^2\big)\langle\phi^R_l,\phi^D_k\rangle\,\cos(\lambda_{R,l}t).
\]
Using
\begin{equation}\label{eq:bateman}
\int_0^\infty \frac12 e^{-\epsilon t}\sin(a t)\cos(b t)\,{\rm d}t
=\frac{a+b}{2\big(\epsilon^2+(a+b)^2\big)}+\frac{a-b}{2\big(\epsilon^2+(a-b)^2\big)}
\quad\text{\cite[(6), p.~19]{batemann}}.
\end{equation}
we obtain, for any $\epsilon>0$,
\[
\int_0^\infty \!\!\frac12 e^{-\epsilon t}\sin(\lambda_{D,k}t)\cos(\lambda_{R,l}t)\,{\rm d}t
=\frac{\lambda_{D,k}+\lambda_{R,l}}{2(\epsilon^2+(\lambda_{D,k}+\lambda_{R,l})^2)}
+\frac{\lambda_{D,k}-\lambda_{R,l}}{2(\epsilon^2+(\lambda_{D,k}-\lambda_{R,l})^2)}.
\]
Taking $\epsilon\to0$ yields
\[
\lim_{\epsilon\to0}\int_0^\infty\!\!\frac12 e^{-\epsilon t}\sin(\lambda_{D,k}t)
\!\left(\int_{\partial \Omega}\mathcal W_R\phi_l^R\,\partial_\nu \overline{\phi^D_k}\,{\rm d}\sigma\right)\!{\rm d}t
= -\,\lambda_{D,k}\,\langle\phi^R_l,\phi^D_k\rangle,
\]
which implies the stated identity.
\end{proof}

\subsection{Inverse problem with Dirichlet boundary condition}\label{sec:dirichlet}
In this subsection, we recover $f\in L^2(\Omega,c(\xx)^{-1}{\rm d}\xx)$ from $\partial_\nu \mathcal W_D f$ on $\partial\Omega\times[0,\infty)$; by completeness of $\{\phi^R_l\}_{l=1}^\infty$ in $L^2(\Omega,c(\xx)^{-1}{\rm d}\xx)$ it suffices to reconstruct $\langle f,\phi_l^R\rangle$.

\begin{thm}\label{thm:co2}
For $f\in L^2(\Omega,c(\xx)^{-1}{\rm d}\xx)$ and every $l\ge 1$, we have
\[
\langle f,\phi^R_l\rangle
= \lambda_{R,l}^{-1}\!
\lim_{\epsilon\to0}\int_0^\infty\!\!\frac{1}{2}\,e^{-\epsilon t}\sin(\lambda_{R,l}t)
\!\left(\int_{\partial \Omega} \partial_\nu \mathcal W_D f(\xx,t)\,\overline{\phi^R_l(\xx)}\,{\rm d}\sigma(\xx)\right)\!{\rm d}t.
\]
\end{thm}

\begin{proof}
It suffices to consider $f=\phi_k^D$. Since $\mathcal W_D\phi_k^D(\xx,t)=\phi_k^D(\xx)\cos(\lambda_{D,k}t)$,
Lemma~\ref{lem:haltmeier} yields
\[
\int_{\partial \Omega}\partial_\nu \mathcal W_D\phi_k^D(\xx,t)\,\overline{\phi^R_l(\xx)}\,{\rm d}\sigma(\xx)
=\big(\lambda_{R,l}^2-\lambda_{D,k}^2\big)\langle\phi^D_k,\phi^R_l\rangle\,\cos(\lambda_{D,k}t).
\]
Applying \eqref{eq:bateman} and letting $\epsilon\to0$ gives
\[
\lim_{\epsilon\to0}\int_0^\infty\!\!\frac12 e^{-\epsilon t}\sin(\lambda_{R,l}t)
\!\left(\int_{\partial \Omega}\partial_\nu \mathcal W_D\phi_k^D\,\overline{\phi^R_l}\,{\rm d}\sigma\right)\!{\rm d}t
=\lambda_{R,l}\,\langle\phi^D_k,\phi^R_l\rangle,
\]
and the claim follows by linearity.
\end{proof}
\subsection{Finite-Time Error Decomposition and Bounds for Boundary-Driven Inversion}
In this section, we quantify the reconstruction error when boundary measurements are available only up to a finite recording time \(T<\infty\). 
\begin{thm}\label{thm:finite}
For $f\in L^2(\Omega,c(\xx)^{-1}{\rm d}\xx)$ , we have
\begin{eqnarray}
&\langle f,\phi^D_k\rangle=-\dfrac{1}{\lambda_{D,k}}\displaystyle \intL^T_0 \intL_{\partial \Omega} \mathcal W_R f(\xx,t)\partial_\nu \overline{\phi^D_k(\xx)}\sin(\lambda_{D,k}t){\rm d}\sigma(\xx){\rm d}t\label{eq:frobin}\\
&\qquad\qquad- \langle\mathcal W_R f(\cdot,T),\phi^D_k\rangle\, \cos(\lambda_{D,k}T)-\dfrac{1}{\lambda_{D,k}}\langle\partial_t \mathcal W_Rf(\cdot, T),\phi^D_k\rangle\,\sin(\lambda_{D,k}T),\nonumber\\
&\langle f,\phi^R_l\rangle=\dfrac{1}{\lambda_{R,l}}\displaystyle \intL^T_0 \intL_{\partial \Omega} \mathcal W_D f(\xx,t)\partial_\nu \overline{\phi^D_l(\xx)}\sin(\lambda_{D,l}t){\rm d}\sigma(\xx){\rm d}t\label{eq:fdirichlet}\\
&\qquad\qquad+ \langle\mathcal W_D f(\cdot,T),\phi^R_l\rangle\, \cos(\lambda_{R,l}T)+\dfrac{1}{\lambda_{R,l}}\langle\partial_t \mathcal W_Df(\cdot, T),\phi^R_l\rangle\,\sin(\lambda_{R,l}T).\nonumber
\end{eqnarray}
\end{thm}
\begin{proof}
Like the proof of Theorem \ref{thm:co1}, Lemma~\ref{lem:haltmeier}  gives
\begin{equation}\label{eq:intT}
\begin{array}{ll}
\displaystyle \intL^T_0 \intL_{\partial \Omega} \mathcal W_R f(\xx,t)\partial_\nu \phi^D_k(\xx)\sin(\lambda_{D,k}t){\rm d}\sigma(\xx){\rm d}t\\
\displaystyle \qquad\qquad=\sum_l \langle f,\phi^R_l\rangle\big(\lambda_{R,l}^2-\lambda_{D,k}^2\big)\langle\phi^R_l,\phi^D_k\rangle\,\intL^T_0\cos(\lambda_{R,l}t)\sin(\lambda_{D,k}t){\rm d}t.
\end{array}
\end{equation}
By simple direct computation, we have
\begin{equation}\label{eq:integralT}
\begin{array}{ll}
\displaystyle \intL^T_0\cos(\lambda_{R,l}t)\sin(\lambda_{D,k}t){\rm d}t\\
\displaystyle =\frac{\lambda_{D,k}-\lambda_{D,k}\cos(\lambda_{R,l}T)\cos(\lambda_{D,k}T)-\lambda_{R,l}\sin(\lambda_{R,l}T)\sin(\lambda_{D,k}T)}{\lambda_{D,k}^2-\lambda_{R,l}^2}.
\end{array}
\end{equation}
Substituting (\ref{eq:integralT}) into (\ref{eq:intT}), we have
\begin{equation*}
\begin{array}{ll}
\displaystyle \intL^T_0 \intL_{\partial \Omega} \mathcal W_R f(\xx,t)\partial_\nu \phi^D_k(\xx)\sin(\lambda_{D,k}t){\rm d}\sigma(\xx){\rm d}t\\
\displaystyle \quad=-\sum_l \langle f,\phi^R_l\rangle\langle\phi^R_l,\phi^D_k\rangle\,\lambda_{D,k}+\sum_l \langle f,\phi^R_l\rangle\langle\phi^R_l,\phi^D_k\rangle\,\lambda_{D,k}\cos(\lambda_{R,l}T)\cos(\lambda_{D,k}T)\\
\displaystyle\qquad\quad+\sum_l \langle f,\phi^R_l\rangle\langle\phi^R_l,\phi^D_k\rangle\,\lambda_{R,l}\sin(\lambda_{R,l}T)\sin(\lambda_{D,k}T)\\
\displaystyle \quad= -\langle f,\phi^D_k\rangle\,\lambda_{D,k}+\langle\mathcal W_R f(\cdot,T),\phi^D_k\rangle\,\lambda_{D,k}\cos(\lambda_{D,k}T)\\
\displaystyle \qquad\quad+\langle\partial_t \mathcal W_Rf(\cdot, T),\phi^D_k\rangle\,\sin(\lambda_{D,k}T),
\end{array}
\end{equation*}
which implies (\ref{eq:frobin}).

For \eqref{eq:fdirichlet}, it suffices to repeat the proof of \eqref{eq:frobin} with the roles of the pairs $\{\lambda_{D,k},\phi_k^{D}\}$ and $\{\lambda_{R,\ell},\phi_\ell^{R}\}$ interchanged.

\end{proof}

As shown in Theorem \ref{thm:finite}, if $\mathcal W_Bf(\xx,T)$ and $\partial_t\mathcal W_Bf(\xx,T)$ are known for $\xx \in\Omega$, then we can recover the initial function $f$ from $\mathcal W_Bf|_{\partial\Omega\times[0,T]}$.

\section{Conclusion}
We addressed the recovery of the initial pressure $f=p(\cdot,0)$ for the variable-speed wave equation on a bounded domain with smooth boundary. Within a unified setting for Dirichlet and Robin boundary conditions, we derived \emph{explicit} identities that recover the spectral coefficients $\langle f,\phi_k^B\rangle$ with respect to the eigenfunction bases of $A_B=-c(\xx)\triangle_{\xx}$, $B\in\{D,R\}$. The formulas are obtained by combining the spectral representation of the wave solution with Green’s identity and a Laplace–sine convolution, and they apply under minimal assumptions: $c\in C^\infty(\bar \Omega)$ with $0<c_m\le c\le c_M$ and a bounded domain $\Omega$ with smooth boundary. Starting from finite linear combinations of eigenfunctions, the identities extend to all $f\in L^2(\Omega,c(\xx)^{-1}{\rm d}\xx)$.

In summary, the paper provides a unified, explicit inversion framework for variable-speed PAT-type models on bounded domains with Dirichlet or Robin boundaries, clarifying how boundary measurements determine the spectral content of the initial pressure and offering a direct path to implementable reconstructions.
\appendix
\section*{Appendix A. Details for Section~\ref{sec:prelim}}
In this section, we determine the orthonormal bases for $L^2(\Omega)$ using the following problem:
\begin{equation}\label{neu}
	\begin{cases}
		-c(\xx)\triangle u(\xx)=f(\xx)\qquad \xx\in \Omega\\ 
		u|_{\partial \Omega}+\alpha\partial_\nu u|_{\partial \Omega}=0 \qquad \alpha \gneq 0.
	\end{cases}
\end{equation}
and
\begin{equation}\label{dir}
	\begin{cases}
-c(\xx)\Delta_{\xx}  u(\xx)=f(\xx) \qquad \xx\in \Omega\\
u|_{\partial \Omega}=0.
\end{cases}
\end{equation}
In the next section, we obtain the coefficients of the initial function with respect to the orthonormal bases.

We introduce the following proposition, which plays a critical role in determining the orthonormal basis: \cite[Propsition 11.34 on page 369]{brezis10}:
\begin{prop}\label{prop:basis}
Let $H$ be a separable Hilbert space over $\mathbb C$ and $T$ be a compact self-adjoint operator on $H$. Then, there exists an orthonormal basis composed of the eigenfunctions of $T$ (and the corresponding eigenvalues are real).
\end{prop}

Now we demonstrate that for any $f\in L^2(\Omega, c(\xx)^{-1}{\rm d}\xx)$, there exists a unique solution $u\in H^1(\Omega)$ to \eqref{neu}, and a unique solution $u\in H_0^1(\Omega)$ to \eqref{dir}. This can be proven through a modification of the standard method
for the case of constant coefficients, which relies on the Lax-Milgram theorem (see \cite[Theorem 1 in Sec.6.2.1]{evans10}). 
To ensure clarity, we will provide an outline of the proof:
First, for \eqref{neu}, we consider the bilinear mapping 
$$
B(u,v)=\int_\Omega\nabla_{\xx} u(\xx)\cdot\nabla_{\xx} v(\xx) {\rm d}\xx+\frac1\alpha\int_{\partial\Omega} \gamma u(\xx)\gamma v(\xx)  {\rm d}\sigma(\xx),\quad u,v\in H^1(\Omega),
$$
where $\sigma(\xx)$ is a measure on the boundary of $\Omega$ and the trace operator $\gamma:H^1(\Omega)\to H^{1/2}(\partial \Omega)$ is continous. 
It is evident that the mappings are well-defined and continuous because $c_m\leq c(\xx)\leq c_M$. Furthermore, the problem \eqref{neu} can be transformed to finding a unique solution $u\in H^1(\Omega)$ satisfying $B(u,v)=b(v)$ for all $v\in H^1(\Omega)$ where $b$ is the linear mapping 
$$b(v)= \int_{\Omega} f(\xx)v(\xx)c(\xx)^{-1}  {\rm d}\xx,\quad v\in H^1(\Omega).$$
By employing the Lax-Milgram theorem, our task simplifies to demonstrating the coerciveness of $B$ on $H^1(\Omega)$, i.e., $B(u,u)\ge C\|u\|_{H^1(\Omega)}$ with some $C>0$ for all  $u\in H^1(\Omega)$.
However, this follows immediately from the known fact that there exists $C>0$ such that  
\begin{equation}\label{funine}
\|\nabla u\|_{L^2(\Omega)}^2+\beta\|u\|_{L^2(\partial\Omega)}^2\geq C\|u\|_{H^1(\Omega)}^2
\end{equation}
for a given $\beta>0$. 
This is a functional inequality belonging to the same family as the classical Poincar\'e inequality
$$\|\nabla u\|_{L^2(\Omega)}\geq C\|u\|_{L^2(\Omega)},\quad u\in H_0^1(\Omega).$$ 
For the Dirichlet case, replace only $B$ by 
$$B(u,v)=\int_\Omega\nabla_{\xx} u(\xx)\cdot\nabla_{\xx} v(\xx) {\rm d}\xx,\quad u,v\in H_0^1(\Omega),$$
and apply the Poincar\'e inequality instead of \eqref{funine}.

Therefore, we can define the operator $T:L^2(\Omega,c(\xx)^{-1}{\rm d}\xx)\to H(\Omega)$ as $Tf=u$, where $u$ is the well-defined solution to \eqref{neu} and \eqref{dir} for $H=H^1$ and $H=H_0^1$, respectively.
Next, in order to use Proposition \ref{prop:basis}, we need to demonstrate that 
$T$ is a compact self-adjoint operator:

\begin{prop}\label{prop:compact}
Let $\Omega\subset \RR^n$ be a bounded domain.
Then the operator $T$ defined above is a compact self-adjoint operator
on $L^2(\Omega,c(\xx)^{-1}{\rm d}\xx)$. 
\end{prop}
\begin{proof}
Note first that $T:L^2(\Omega,c(\xx)^{-1}{\rm d}\xx)\to H(\Omega)$ is bounded.
Indeed, when $H=H^1$, from $B(u,u)=b(u)$ and H\"older's inequality, 
$$
\|\nabla_{\xx}u\|_{L^2(\Omega)}^2+\frac1\alpha\|u\|_{L^2(\partial\Omega)}^2
\leq C\|f\|_{L^2(\Omega,c(\xx)^{-1}\rm{d}\xx)}\|u\|_{L^2(\Omega)}.
$$
Combining this with \eqref{funine} implies
the boundedness 
$$\|Tf\|_{H^1(\Omega)}=\|u\|_{H^1(\Omega)}\leq C\|f\|_{L^2(\Omega,c(\xx)^{-1}\rm{d}\xx)}$$
as desired.
Similarly for $H=H_0^1$.	
	
By the boundedness above and the compactness of the injection 
$H(\Omega)\subset L^2(\Omega)$, we then conclude that 
	$T:L^2(\Omega,c(\xx)^{-1}{\rm d}\xx)\to L^2(\Omega)$  is a compact operator.
Therefore, $T:L^2(\Omega,c(\xx)^{-1}{\rm d}\xx)\to L^2(\Omega,c(\xx)^{-1}{\rm d}\xx)$
is also compact.
It is not difficult to show that $T$ is linear and self-adjoint by the definition of $T$.
The proof is complete.
\end{proof}

Consequently, Proposition \ref{prop:compact} guarantees that the orthonormal basis of $L^2(\Omega,c(\xx)^{-1}{\rm d}\xx)$ is composed of the eigenfunctions of $T$ (equivalently, of $-c(\xx)\Delta_{\xx} $) with respect to \eqref{neu} (or \eqref{dir}).
Now, let $\{\phi^R_l \}$ be the eigenfunctions of $-c(\xx)\Delta_\xx$ with eigenvalue, say $\lambda_{R,l}^2$, that is,
$$ c(\xx)\Delta_\xx\phi^R_l(\xx)+\lambda^2_{R,l}\phi^R_l(\xx)=0 \mbox{ on } \Omega
\quad\mbox{and}\quad \phi^R_l|_{\partial \Omega}+\alpha\partial_\nu\phi^R_l|_{\partial \Omega}=0.$$
The eigenvalues here must be positive; let $-c(\xx)\Delta_{\xx}  \phi(\xx)=\Lambda \phi(\xx)$.
Then,
\begin{align*}
\int_\Omega|\nabla\phi(\xx)|^2 {\rm d}\xx
&= \int_\Omega-(\Delta_{\xx} \phi(\xx))\phi(\xx) {\rm d}\xx
+\int_{\partial\Omega} \partial_\nu\phi(\xx)\phi(\xx) {\rm d}\sigma(\xx)\\
&=\Lambda\int_\Omega\phi^2(\xx)c(\xx)^{-1}{\rm d}\xx
-\frac1\alpha\int_{\partial\Omega}\phi^2(\xx){\rm d}\sigma(\xx).
\end{align*}
Hence, $\Lambda$ is positive. (In fact, $\Lambda\geq \alpha c_M^{-1}$.)

Finally, $ \{\phi^R_l \}_{l=1}^\infty$ is the orthonormal basis of $L^2(\Omega,c(\xx)^{-1}{\rm d}\xx)$ and we have
$$
f=\sum_{l=1}^\infty <f,\phi^R_l>\phi^R_l \quad\mbox{where}\quad<f,g>=\int_{\Omega}f(\xx)\overline{g(\xx)}c(\xx)^{-1}{\rm d}\xx.
$$
Similarly, we denote by $ \{\phi^D_k \}_{k=1}^\infty$ the orthonormal basis of $L^2(\Omega,c(\xx)^{-1}{\rm d}\xx)$ with respect to \eqref{dir}, 
and
$$
f=\sum_{k=1}^\infty <f,\phi^D_k>\phi^D_k.$$




\bibliographystyle{plain}


\end{document}